\documentclass{amsart}

\usepackage{amssymb,amsfonts} 	
\usepackage{amsmath}            
\usepackage{amsthm}
\usepackage{bbm}

\newtheorem{theorem}{Theorem}[section]

\newtheorem{prop}[theorem]{Proposition}

\theoremstyle{definition}

\theoremstyle{remark}

\numberwithin{equation}{section}

\renewcommand{\le}{\leqslant}
\renewcommand{\ge}{\geqslant}
\renewcommand{\Xi}{\varXi}

\newcommand{\eps}{\varepsilon}

\newcommand{\Z}{\mathbb{Z}}

\newcommand{\N}{\mathbb{N}}

\newcommand{\bO}{\mathcal{O}}

\newcommand{\abs}[1]{\left| #1 \right|}

\DeclareMathOperator{\e}{e}

\begin{document}

\title[On a problem of Chen and Liu]{On a problem of Chen and Liu concerning the prime power factorization of $n!$}

\author{Johannes F. Morgenbesser}
\address{Institut f\"ur Diskrete Mathematik und Geometrie, Technische Universit\"at Wien, Wiedner Hauptstra\ss e 8--10, A--1040 Wien, Austria,
}
\email{johannes.morgenbesser@tuwien.ac.at}
\thanks{The first author was supported by the Austrian Science Foundation FWF, grant S9604, that is part of the National Research Network ``Analytic Combinatorics and Probabilistic Number Theory''.}

\author{Thomas Stoll}
\address{Institut de Math\'ematiques de Luminy, Universit\'e d'Aix-Marseille, 13288 Marseille Cedex 9, France,}
\email{stoll@iml.univ-mrs.fr}
\thanks{This research was supported by the Agence Nationale de la Recherche, grant ANR-10-BLAN 0103 MUNUM}



\date{\today}


\commby{}

\begin{abstract}
For a fixed prime $p$, let $e_p(n!)$ denote the order of $p$ in the prime factorization of $n!$. Chen and Liu (2007) asked whether for any fixed $m$, one has
$\{e_p(n^2!) \bmod m:\; n\in\mathbb{Z}\}=\mathbb{Z}_m$ and $\{e_p(q!) \bmod m:\; q \mbox{ prime}\}=\mathbb{Z}_m$. We answer these two questions and show asymptotic formulas for $ \# \{n<x: n \equiv a \bmod d,\; e_p(n^2!)\equiv r \bmod m\}$ and $ \# \{q<x: q \mbox { prime}, q \equiv a \bmod d,\; e_p(q!)\equiv r \bmod m\}$. Furthermore, we show that for each $h\geq 3$,
we have $\{n<x: n \equiv a \bmod d,\; e_p(n^h!)\equiv r \bmod m\} \gg x^{4/(3h+1)}$.
\end{abstract}

\maketitle

\section{Introduction}\label{intro}

Let $p_1=2$, $p_2=3$, $\ldots$ be the sequence of prime numbers in ascending order and consider the prime factorization of 
$$n!=\prod_{p_j\le n} p_j^{e_{p_j}(n!)}.$$
Legendre~\cite[p.10--12]{Le30} (see also~\cite[p.~263]{Di05},~\cite[Ch.~1.3]{Te08}) showed that for any nonnegative integer $n$ and any fixed prime $p$ we have
\begin{equation}\label{legendre}
  e_p(n!)=\sum_{i\geq 1}\left\lfloor\frac{n}{p^i}\right\rfloor=\frac{n-s_p(n)}{p-1},
\end{equation}
where $s_p(n)$ denotes the sum of the digits of $n$ in base $p$, \textit{i.e.},
$$s_p(n)=\sum_{i\geq 0} \varepsilon_i(n),\qquad \mbox{for }\quad n=\sum_{i\geq 0} \varepsilon_i(n) p^i,$$
where $\varepsilon_i(n)\in\{0,1,\ldots,p-1\}$. A well-known area of application for $e_p(n!)$ is the determination of the explicit numerical error term in Mertens first theorem~\cite[Ch.~1.4]{Te08}. The investigation of the distribution properties of $e_p(n!)$ can be said to have started with Erd\H os and Graham~\cite[p.77]{EG80} who stated (in our notation) that \textit{``it is annoying that we cannot even show that for all $k$ there is an $n_k$ so that in the prime decomposition of $n_k!$ all the $e_{p_j}(n_k!)$, $1\le j\le k$, are even.''} In 1997, Berend~\cite{Be97} solved this problem by showing that for any fixed $m\geq 2$ there are infinitely many $n$ that satisfy
$$e_{p_1}(n!)\equiv e_{p_2}(n!)\equiv \cdots \equiv e_{p_k}(n!)\equiv 0 \bmod m,$$
and the set of all such $n$ has bounded gaps. In his solution, Berend~\cite{Be97} strengthened the problem of Erd\H os and Graham in two different directions. On the one hand, he  not only considered the parity of the exponents but studied more generally if they were divisible by a fixed integer $m\ge2$. On the other hand, he
already treated subsets of integers with prescribed multiplicative properties instead of looking at the entire set of integers $n$. In particular, he showed that for arbitrary fixed positive $D, k$ and $m$ there exist infinitely many $n$ such that all the exponents $e_{p_j}((dn)!),\; 1\le j\le k, \, 1\le d\le D,$
are divisible by $m$. 

Several authors considered in the last years extensions of the Erd\H os-Graham problem, namely, Berend/Kolesnik~\cite{BK07}, Chen~\cite{Ch03}, Chen/Liu~\cite{CL06,CL07}, Chen/Zhu~\cite{CL00}, Luca/St{\u{a}}nic{\u{a}}~\cite{LS03}, Sander~\cite{Sa01} and Zhai~\cite{Zh09}. The most general result is due to Berend and Kolesnik~\cite{BK07} who proved unconditionally that
\begin{align*}
&  \#\{0\le n<x : n \equiv a \bmod d,\; e_{q_j}(n!)\equiv r_j \bmod m_j, \;\, 1\le j\le k\}\\
&\hspace{7.0cm}  =\frac{x}{d m_1 m_2\cdots m_k}+\bO\left(x^{1-\delta}\right),\label{BKresult}
\end{align*}
for any integer $a$ and  $d\ge1$ where $k\geq 1$ is fixed, $\mathbf{q}=(q_1, q_2, \ldots, q_k)$ is a vector of distinct, not necessarily ordered primes, $\mathbf{m}=(m_1, m_2, \ldots, m_k)$ is a vector of arbitrary integers $\geq 2$, and $\mathbf{r}=(r_1, r_2, \ldots, r_k)$ is such that $0\le r_j< m_j$ for $j=1,2,\ldots,k$, and $\delta=\delta(\mathbf{m},\mathbf{q},\mathbf{r})>0$ is effectively computable. 

\medskip

%

Intriguing problems arise when the sequence of integers $n$ lying in a fixed residue class is replaced by sparser sequences such as primes, squares or higher-degree powers. Chen and Liu~\cite{CL07} posed several problems in that respect (see also~\cite{Zh09} for generalizations of these problems). In particular, at the end of their paper they remark that they even have no answer to the following basic questions: 
\bigskip

\noindent \textbf{Question~1:} Is it true that for all fixed $p$ and $m$, 
\begin{equation*}
  \{e_p(n^2!) \bmod m:\; n\in\mathbb{Z}\}=\mathbb{Z}_m\; ?
\end{equation*}

\noindent \textbf{Question~2:} Is it true that for all fixed $p$ and $m$, 
\begin{equation*}
  \{e_p(q!) \bmod m:\; q \mbox{ prime}\}=\mathbb{Z}_m\; ?
\end{equation*}

\medskip

Zhai~\cite[Theorems~3 and~4]{Zh09} obtained a partial answer to Question~1. He showed
that for all $h\geq 2$ and $r\in \Z$, there are infinitely many $n$ such that  $e_p(n^h!)\equiv r \bmod m$
provided that
\begin{equation}\label{zhaicondition}
  p\geq \begin{cases}
  4m-2, & \mbox{if } h=2,\\
  h^h m^{h-1}, & \mbox{if } h\geq 3.
  \end{cases}
\end{equation}
From his proof one can obtain a lower bound of the form\footnote{By $f\ll_\omega g$ resp. $f\gg_\omega g$ we mean that there exists a constant $C$ depending at most on $\omega$ such
that $f\le C g$ resp. $f\geq C g$.}
$$ \#\{n<x :\; e_p(n^h!)\equiv r \bmod m\}\gg_{p,h,m} \log x, \qquad x\to \infty.$$
Unfortunately, Zhai's method cannot be applied in the case of small $p$, such as to treat $e_2(n^2!)$ or $e_5(n^2!)$. 

\medskip

The aim of the present paper is to use our current knowledge of the distribution properties of the sum-of-digits function
to give complete answers to Questions~1 and~2. We are able to improve on Zhai's result and to generalize Chen and Liu's questions
in two different respects. First, we are able to drop the superfluous condition~(\ref{zhaicondition}) and to find asymptotic formulas for the counting
functions in the case of squares and primes. Second, we give a general lower bound for $h\geq 3$. 

Using our results we get the following nice application: Let $Z(n)$ be the number of ending 0's in base $10$ of $n!$. Observe that $$Z(n)=\min\{e_{2}(n!),e_{5}(n!)\}=e_{5}(n!).$$
Then it will follow from Theorem~\ref{thm1} that
  $$\lim_{x\to \infty} \frac{1}{x}\#\{ n<x : Z(n^2) \equiv r \bmod m\} = \frac{1}{m}$$
for every $m\ge 2$ and $0\le r < m$. The analogous result holds also true for the number of ending 0's of factorials of primes. 

\section{Main results}

In the sequel, let $\pi(x;a,d)$ be the number of primes $\equiv a\bmod d$ that are less than or equal to $x$. 

\begin{theorem}\label{thm1}
 Let $p$ be a prime, $m,d\geq 1$ and $0\le a<d$, $0\le r < m$. Then there exist constants $\delta^{(1)}_{p,m} >0$ and $\delta^{(2)}_{p,m} >0$ such that
\begin{align*}
 &\#\{ n<x : n\equiv a \bmod d,\, e_p(n^2!) \equiv r \bmod m\} \\
 &\qquad\qquad\qquad\qquad= \frac{x}{dm} + \bO\left((\log x)^{11/4}  x^{1-\delta^{(1)}_{p,m}}\right),
\end{align*}
and
\begin{align*}
 &\#\{ q<x : q \mbox{ prime, } q\equiv a \bmod d,\,  e_p(q!) \equiv r \bmod m\}  \\
 &\qquad\qquad\qquad\qquad= \frac{\pi(x;a,d)}{m} + \bO\left( (\log x)^3 x^{1-\delta^{(2)}_{p,m}}\right).
\end{align*}
The implied constants depend only on $p$.
\end{theorem}

The proof of this result is notably based on recent work by Mauduit and Rivat~\cite{MR09} and Martin, Mauduit and Rivat~\cite{MMRpre}, and uses exponential sum estimates of hybrid type. In contrast, we use an idea of Stoll~\cite{St11pre} to obtain general lower bounds for higher-degree powers. The method is constructive.

\begin{theorem}\label{thm2}
  Let $h\geq 2$, $p$ be a prime, $m, d\geq 1$ and $0\le a<d$, $0\le r<m$. Then, as $x\to \infty$,
  \begin{equation}\label{genbound}
    \#\{n<x : n \equiv a \bmod d,\; e_p(n^h!)\equiv r \bmod m\}\gg_{p,h,d,m} x^{4/(3h+1)}.
  \end{equation}
  Moreover, there is an effectively computable constant $C=C(p,h,d,m)$ such that
  $$ \{e_p(n^h!) \bmod m:\; 0\le n < C, \; n \equiv a \bmod d\}=\mathbb{Z}_m.$$
\end{theorem} 

\medskip

The constant $C$ can be directly obtained from the proof. We remark that
\begin{equation}\label{Zm2}  
  \{e_p(n^h!) \bmod m:\;  0\le n < p^{1/h}+(m-2)d,\; n\equiv a \bmod d\}\neq \mathbb{Z}_m.
\end{equation}
By a probabilistic argument one might expect that we have the full set of residues after about $m \log m$ steps.
However, as~(\ref{Zm2}) shows, this is not true since there is a crucial dependency of $p$ in the bound for $n$ in~(\ref{Zm2}).

\section{Proof of Theorem~\ref{thm1}}

Legendre's formula~(\ref{legendre}) shows that
\begin{equation}\label{nsqu1}
 e_p(n!) \equiv r \bmod m \quad \Longleftrightarrow \quad n - s_p(n) \equiv r(p-1) \bmod (p-1)m.
\end{equation}
In order to prove Theorem~\ref{thm1}, we need some auxiliary results. In particular, we have to deal with exponential sums containing the sum-of-digits function of primes and squares. Let $\omega(b)$ denote the number of different prime divisors of $b$. The first proposition
is a generalization of~\cite[Theorem 1]{MR09} and the second proposition is taken from~\cite[Proposition 4]{MMRpre}.

\begin{prop}\label{prop1}
Let $b\ge 2$ and $\alpha,\beta,\gamma$ real numbers such that $(b-1)\alpha\not\in\Z$. Then there exists a constant $\sigma^{(1)}_{b,\alpha}>0$  such that
\footnote{If $\beta=\gamma=0$, \cite[Theorem 1]{MR09}  shows this result  with an error term of the form $(\log x)^{(\omega(b) +8)/2}  x^{1-\sigma}$ instead of $(\log x)^{(\omega(b) +10)/4} x^{1-\sigma}$. However, we want to remark that the proof given in~\cite{MR09} already implies the better error term as stated in this proposition.}

\[
 \sum_{n<x} \e(\alpha s_b(n^2) + \beta n^2+ \gamma n) \ll_b (\log x)^{(\omega(b) +10)/4} x^{1-\sigma^{(1)}_{b,\alpha}}.
\]
\end{prop}

\begin{proof}
This result can be proven in the same way as~\cite[Theorem 1]{MR09}. Thus, we just give a short outline. Let $b^{\nu-1} < x \le b^\nu$ and set $f(n) = \alpha s_b(n)$. As in the Mauduit--Rivat case, it suffices to show that
\begin{align}\label{eq:nsqu2}
 S_1:= \sum_{b^{\nu-1} <n \le x} \e(f(n^2) + \beta n^2+ \gamma n) \ll_b (\log x)^{(\omega(b) + 6)/4} x^{1-\sigma^{(1)}_{b,\alpha}}
\end{align}
for some constant $\sigma^{(1)}_{b,\alpha}$.
Lemma 15 from~\cite{MR09} (a van der Corput-type inequality) implies that $S_1$ is bounded by (some constant times)
\begin{align*}
&b^{\nu - \rho/2} +b^{\nu/2}   \max_{1\le |r|< b^\rho}  \Bigg|\sum_{b^{\nu-1} <n \le b^\nu }\!\!\! \e(f((n+r)^2) + \beta (n+r)^2 + \gamma(n+r))\\
&\hspace{4cm} \cdot \e( - f(n^2) - \beta n^2 -\gamma n)\Bigg|^{1/2}\\
&\ll_b b^{\nu - \rho/2} + b^{\nu/2} \max_{1\le |r|< b^\rho} \abs{\sum_{b^{\nu-1} < n \le b^\nu} \e(f((n+r)^2)  - f(n^2) + 2 \beta nr ) }^{1/2}\!\!,
\end{align*}
where $1\le \rho \le \nu/2$ is an integer which we will choose later on. Set $\lambda:= \nu+2\rho+1$. Using~\cite[Lemma 16]{MR09}, we obtain
\begin{align}\label{eq:nsqu3}
 S_1 \ll_b b^{\nu-\rho/2} + b^{\nu/2} \max_{1\le |r|<b^\rho} \abs{ S_2}^{1/2},
\end{align}
where
\[
 S_2 := \sum_{b^{\nu-1} < n \le b^\nu} \e( f_\lambda((n+r)^2)  - f_\lambda(n^2) + 2\beta nr ),
\]
and $f_\lambda(n)$ is defined by
\[
 f_\lambda(n) := \alpha \sum_{0\le j<\lambda} \eps_j(n),
\]
where $\eps_j(n)$ denotes the $j$-$th$ digit of $n$.
Note, that $f_\lambda(n)$ (a so-called truncated sum of digits function) sums up just the $\lambda$ lower placed digits (multiplied with $\alpha$).
Lemma 17 from~\cite{MR09} (again a van der Corput-type inequality) implies now that
\begin{align}\label{eq:nsqu4}
 |S_2|^2 \le b^{2\nu- 2\rho} + b^\nu \max_{1\le |s|<b^{2\rho}} |S_3|,
\end{align}
where
\begin{align*}
 S_3 &= \!\!\!\sum_{I(\nu,s,\mu)} \!\!\!\e( f_\lambda((n+r+sb^\mu)^2) - f_\lambda((n+sb^\mu)^2) + 2\beta(n+sb^\mu)r)\\
&\hspace{4cm}\cdot \e( - f_\lambda((n+r)^2) + f_\lambda(n^2) - 2\beta nr),
\end{align*}
the interval $I(\nu,s,\mu)$  is given by $I(\nu,s,\mu) = \{ n\in \N : b^{\nu-1} < n, n + sb^\mu \le b^\nu \}$ and $\mu$ is an integer satisfying $1\le \mu \le \nu - 2\rho-1$. 
Thus we get that $|S_3|$ is equal to
\begin{align*}
\abs{ \sum_{I(\nu,s,\mu)}\!\!\!\! \e( f_\lambda((n+r+sb^\mu)^2) - f_\lambda((n+r)^2) - f_\lambda((n+sb^\mu)^2) + f_\lambda (n^2) ) }
\end{align*}
Note, that the terms containing $\beta$ and $\gamma$ are vanished. Mauduit and Rivat considered exactly the term $S_3$ 
and they showed that
\begin{align}\label{eq:nsqu5}
 |S_3| \ll_b \nu^{\omega(b) + 6} b^{\nu-2\rho}
\end{align}
for every $1\le\rho\le \nu/2$, $1\le \mu \le \nu-2\rho -1$, $1\le |r|<b^\rho$ and $1\le |s|< b^{2\rho}$ (see~\cite[Eq.~(45)]{MR09}).
Equations~\eqref{eq:nsqu3},~\eqref{eq:nsqu4}, and~\eqref{eq:nsqu5} finally imply
\[
 S_1 \ll_b \nu^{(\omega(b)+6)/4} b^{\nu-\rho/2}.
\]
As in~\cite{MR09}, it is now possible to choose $\rho$ and $\mu$ in order to obtain~\eqref{eq:nsqu2}. This finishes the proof of Proposition~\ref{prop1}.
\end{proof}

\begin{prop}\label{prop2}
 Let $b\ge 2$ and $\alpha,\beta$ real numbers such that $(b-1)\alpha\not\in\Z$. Then there exists a constant $\sigma^{(2)}_{b,\alpha}>0$  such that
\[
 \sum_{\substack{q<x\\q \mbox{ \tiny prime}}} \e(\alpha s_b(q) + \beta q) \ll_b (\log x)^3 x^{1-\sigma^{(2)}_{b,\alpha}}.
\]
\end{prop}


\begin{proof}[Proof of Theorem~\ref{thm1}]
We just give a proof of the stated result for the squares. The case $e_p(q!)$, $q$ prime, can be shown exactly the same way but using Proposition~\ref{prop2} instead of Proposition~\ref{prop1}. 
%
In the following we use the abbreviation 
\[
m' = (p-1)m.
\]
Relation~(\ref{nsqu1}) allows us to write
\begin{align*}
\#\{ n<x : n\equiv a \bmod d,\, e_p(n^2!) \equiv r \bmod m\} =  \sum_{0\le j<m'} T_j(x),
\end{align*}
where 
\[
T_j(x):=\# \{ n<x : n \equiv a \bmod d,\, n^2 \equiv j \bmod m',\; s_p(n^2) \equiv j - r(p-1) \bmod m'\}.
\]
Using discrete Fourier analysis, we have
\begin{align*}
T_j(x)& = \sum_{n< x} \left(\frac{1}{d} \sum_{0\le u < d} \e\left( u \frac{n-a}{d}\right)\right)\!\cdot\! \left(\frac{1}{m'} \sum_{0\le k < m'} \e\left( k \frac{n^2-j}{m'}\right)\right)\\
&\hspace{2.5cm} \cdot
\left(\frac{1}{m'} \sum_{0\le \ell < m'} \e\left( \ell \frac{s_p(n^2)-(j-r(p-1))}{m'}\right)\right).
\end{align*}
This can be written as
\begin{align*}
T_j(x)& = \frac{1}{d m'^2 }   \sum_{n< x}   \sum_{0\le u < d} \e\left(u \frac{n- a}{d}\right) \sum_{0\le k < m'} \e\left( k \frac{n^2-j}{m'}\right)\\
&\hspace{2.5cm} \cdot \sum_{\substack{0\le \ell_1 < p-1\\ 0\le \ell_2 < m}} \e\left( (\ell_1 m + \ell_2) \frac{s_p(n^2)-(j-r(p-1))}{m'}\right)\!,
\end{align*}
and we obtain (splitting the part coming from $\ell_2=0$ and $\ell_2>0$)  
\begin{align*}
T_j(x) &= \frac{1}{m'^2} \sum_{\substack{n<x\\ n \equiv a \bmod d}} \sum_{0\le k <m'} \sum_{0\le \ell_1<p-1} \e\left( \frac{kn^2-kj - \ell_1 m j + \ell_1 m s_p(n^2)}{m'} \right) \\
& \;+ \bO \left( \frac{1}{d m'^2} \sum_{\substack{0 \le u < d\\0\le k < m'}} \sum_{\substack{0\le \ell_1 < p-1\\ 0 < \ell_2 < m}} \abs{\sum_{n< x} \e \left( \frac{\ell_1 m + \ell_2 }{m'}s_p(n^2) + \frac{k}{m'}n^2 + \frac{u}{d}n \right)}\right)\!.
\end{align*}
Thus we get that $\#\{ n<x : n \equiv a \bmod d,\, e_p(n^2!) \equiv r \bmod m\}$ is given by
\begin{align}\label{eq:nsqu1}
 MT +  \bO \left(\max_{\substack{0\le u < d\\0\le k <m'\\ 0\le \ell_1<p-1\\0<\ell_2 < m}}\abs{ \sum_{n< x} \e \left( \frac{\ell_1 m + \ell_2 }{m'}s_p(n^2) + \frac{k}{m'}n^2+ \frac{u}{d}n \right)} \right),
\end{align}
where 
\[
MT:=\frac{1}{m'^2} \sum_{0\le j < m'} \sum_{\substack{n<x\\ n\equiv a \bmod d}} \sum_{0\le k <m'} \sum_{0\le \ell_1<p-1} \e\left( \frac{kn^2-kj - \ell_1 m j + \ell_1 m s_p(n^2)}{m'} \right).
\]
Next we calculate the main term $MT$ in~\eqref{eq:nsqu1}. Therefore, let us define $\mathbbm{1}_j(n)$ for all $0\le j < m'$ and for all positive integer $n$ by
\[
 \mathbbm{1}_j(n) =
\begin{cases}
1, & \mbox{ if $n \equiv j \bmod m'$,}\\
0, & \mbox{ otherwise.}
\end{cases}
\]
Then we get that the main term $MT$ is equal to 
\begin{align*}
&\frac{1}{m'} \sum_{0\le \ell_1 < p-1} \sum_{\substack{n<x\\ n \equiv a \bmod d}} \sum_{0\le j < m'} \e\left(\frac{- \ell_1 m j + \ell_1 m s_p(n^2)}{m'} \right)\frac{1}{m'} \sum_{0\le k < m'} \e\left( k\frac{n^2-j }{m'} \right)\\
&= \frac{1}{m'} \sum_{0\le \ell_1 < p-1} \sum_{\substack{n<x\\ n \equiv a \bmod d}} \sum_{0\le j < m'} \e\left(\ell_1\frac{s_p(n^2)-j}{p-1} \right) \cdot \mathbbm{1}_j(n^2).
\end{align*}
Since $s_p(n^2)\equiv j \bmod (p-1)$ if $n^2 \equiv j \bmod m'$, we obtain that the remaining exponential part is equal to $1$ for all nonzero summands. Furthermore, the relation
\[
 \sum_{0\le j < m'} \mathbbm{1}_j(n) = 1
\]
holds trivially for any integer $n$. Thus we finally have
\begin{align*}
 MT &= \frac{1}{m'} \sum_{0\le \ell_1 < p-1} \sum_{\substack{n<x\\ n \equiv a \bmod d}} \sum_{0\le j < m'} \mathbbm{1}_j(n^2)\\
&= \frac{p-1}{m'} \sum_{\substack{n<x\\ n \equiv a \bmod d}} \sum_{0\le j < m'} \mathbbm{1}_j(n^2) = \frac{x}{dm} + \mathcal{O}(1).
\end{align*}

It remains to bound the error term in~\eqref{eq:nsqu1}.
Since $0<\ell_2 <m$, we have
$(\ell_1 m + \ell_2) / m' \cdot (p-1) = \ell_1 + \ell_2/m \not\in\Z$. Thus we can employ Proposition~\ref{prop1} (note, that $\omega(p)=1$ since $p$ is prime). Setting 
\[
 \delta^{(1)}_{p,m} := \min_{\substack{0\le \ell_1 < p-1\\0 < \ell_2 < m}} \sigma^{(1)}_{p, (\ell_1 m +\ell_2)/m'},
\]
we finally obtain the desired result.
\end{proof}

\section{Proof of Theorem~\ref{thm2}}

\begin{proof}[Proof of Theorem~\ref{thm2}]

Consider the polynomial $t(x)\in\Z[x]$ with
$$t(x)=d(p-1)m\cdot\left(m_3x^3+m_2x^2-m_1x+m_0\right)+a,$$
where $m_3, m_2, m_1, m_0$ are positive integers. Lemma 2.1 in~\cite{St11pre} says that for all $u\geq 1$ and
\begin{align}
  p^{u-1} &\le m_0+\frac{a}{d(p-1)m} <p^u,\nonumber\\
  p^{u-1} &\le m_2, m_3 <p^u,\nonumber\\
  1&\le m_1<p^u/\left(hp (6p)^h\right)\label{ranges}
\end{align}
the polynomial $(t(x))^h=\sum_{i=0}^{3h} c_ix^i\in\Z[x]$ has all positive integral coefficients with the only exception of the coefficient of $x^1$ which is negative. Also, note that 
$a\le d(p-1)m $ and thus we have
$$|c_i|\le \left(4p^ud(p-1)m\right)^h.$$
In order to have that the range~(\ref{ranges}) for $m_1$ contains at least one admissible integer $m_1$ we suppose now that $u$ is such that
\begin{equation}\label{ubound}
  p^u> hp(6p)^h.
\end{equation}
Furthermore, let $k$ be such that
\begin{equation}\label{kbound}
  p^k>\left(4p^ud(p-1)m\right)^h.
\end{equation}
Note that $p^k   \gg_{p,h,d,m} p^{uh}$ as $u\to \infty$. We get as in~\cite{St11pre} that
$$s_p((t(p^k))^h) = k(p-1)+M,$$
where $M$ does not depend on $k$ provided $k$ is such as in~(\ref{kbound}). In addition, we have $t(p^k)\equiv a \bmod d$ and $(t(p^k))^h\equiv a^h \bmod (p-1)m$. Therefore, by~(\ref{nsqu1}), for each $k$ with~(\ref{kbound}) and $j\geq 0$ we have
\begin{align*}
  e_p((t(p^{k+j}))^h!) &=\frac{(t(p^{k+j}))^h-s_p((t(p^{k+j}))^h)}{p-1}\\
  &\equiv \frac{a^h}{p-1}-(k+j)-\frac{M}{p-1} \bmod m.
\end{align*}
Note that $(p-1)\vert (a^h-M)$ so that for each fixed $r$ with $0\le r<m$ there is $j$ with $0\le j \le m-1$ such that
\begin{align*}
  e_p((t(p^{k+j}))^h!) &\equiv r \bmod m.
\end{align*}
By construction we thus find $\gg_{p,h,d,m} p^{4u}$ distinct integers that are all $\ll_{p,h,d,m} p^{u(3h+1)}$ (for more details we refer to~\cite{St11pre}). This proves~(\ref{genbound}).

To get an explicit bound for $C(p,h,d,m)$ we only have to make some admissible choices, say, $u_0$ and $k_0$,  for $u$ in~(\ref{ubound}) resp. $k$ in~(\ref{kbound}), and estimate $t(p^{k+m-1})$. First we take $u_0$ to be such that
$p^{u_0}\le hp^2(6p)^h.$ Secondly, we find $k_0$ such that $p^{k_0}\le \left(4p^{u_0}d(p-1)m\right)^h p.$
It is now a straightforward calculation to find
\begin{align*}
  C \le t(p^{k+m-1})&\le 2(p-1)mp^{u_0}p^{3(k_0+m-1)}\\
  &\leq m^{3h+1}d^{3h}(p-1)^{3h+1} p^{3m} \left(4h p^2 (6p)^h\right)^{1+3h}. 
\end{align*}
This concludes the proof of Theorem~\ref{thm2}.
\end{proof}

\section{Concluding remarks}

We end our discussion with a few remarks. It seems possible to use the approach of Drmota, Mauduit and Rivat~\cite{DMR11} to get an asymptotic formula in Theorem~\ref{thm2} provided that $p$ is a \textit{very large} prime whose size is about exponential in $h$. For \textit{small} $p$ it is already an open and surely very difficult problem to find an asymptotic formula for $e_p(n^3)$ in arithmetic progressions. As a further remark, we also stress the fact that Theorem~\ref{thm1} and Theorem~\ref{thm2} hold for arbitrary quadratic polynomials in place of $n^2$, respectively, for arbitrary $P(x)\in\Z[x]$ of degree $h$ (with $P(\N)\subset \N$) in place of $x^h$. A minor variation of the used arguments will yield these generalizations.  


\end{document}